\documentclass[12pt]{amsart}

\usepackage{amssymb}
\usepackage{bm}
\usepackage{graphicx}
\usepackage{psfrag}
\usepackage{hyperref}
\hypersetup{colorlinks=true, linkcolor=blue, citecolor=red, urlcolor=wine}
\usepackage{amscd}
\usepackage{color}
\usepackage{url}
\usepackage{algpseudocode}
\usepackage{fancyhdr}
\usepackage{xy}
\input xy
\xyoption{all}
\usepackage{stmaryrd}

\usepackage[dvipsnames]{xcolor}

\newtheorem*{maintheorem*}{Main Theorem}
\newtheorem{theorem}{Theorem}[section]
\newtheorem{prop}[theorem]{Proposition}
\newtheorem{lemma}[theorem]{Lemma}
\newtheorem{cor}[theorem]{Corollary}

\theoremstyle{definition}
\newtheorem{definition}[theorem]{Definition}
\newtheorem{remark}[theorem]{Remark}
\newtheorem{example}[theorem]{Example}

\numberwithin{equation}{section}

\newcommand{\ff}{\mathbb{F}}

\newcommand{\nn}{\mathbb{N}}

\newcommand{\qq}{\mathbb{Q}}
\newcommand{\rr}{\mathbb{R}}
\newcommand{\zz}{\mathbb{Z}}

\newcommand{\cont}{\mathsf{c}}
\providecommand\ldb{\llbracket}
\providecommand\rdb{\rrbracket}

\keywords{monoid rings, monoid algebras, Gauss's Lemma, Eisenstein's Criterion, Puiseux algebras, atomic domains, half-factoriality, other-half-factorial monoids, Puiseux monoids, numerical semigroups}

\begin{document}
	
	\mbox{}
	\title{Irreducibility and factorizations \\  in monoid rings}
	
	\author{Felix Gotti}
	\address{Department of Mathematics\\UC Berkeley\\Berkeley, CA 94720 \newline \indent Department of Mathematics\\Harvard University\\Cambridge, MA 02138}
	\email{felixgotti@berkeley.edu}
	\email{felixgotti@harvard.edu}
	
	\date{\today}
	
	\begin{abstract}
		For an integral domain $R$ and a commutative cancellative monoid $M$, the ring consisting of all polynomial expressions with coefficients in $R$ and exponents in $M$ is called the monoid ring of $M$ over $R$. An integral domain is called atomic if every nonzero nonunit element can be written as a product of irreducibles. In the investigation of the atomicity of integral domains, the building blocks are the irreducible elements. Thus, tools to prove irreducibility are crucial to study atomicity. In the first part of this paper, we extend Gauss's Lemma and Eisenstein's Criterion from polynomial rings to monoid rings. An integral domain $R$ is called half-factorial (or an HFD) if any two factorizations of a nonzero nonunit element of $R$ have the same number of irreducible elements (counting repetitions). In the second part of this paper, we determine which monoid algebras with nonnegative rational exponents are Dedekind domains, Euclidean domains, PIDs, UFDs, and HFDs. As a side result, we characterize the submonoids of $(\qq_{\ge 0},+)$ satisfying a dual notion of half-factoriality known as other-half-factoriality.
	\end{abstract}
	\bigskip
	
	\maketitle
	
\section{Introduction}
\label{sec:intro}

Given an integral domain $R$ and a commutative cancellative monoid $M$, the ring of all polynomial expressions with coefficients in $R$ and exponents in $M$ is known as the monoid ring of $M$ over $R$ (cf. group rings). Although the study of group rings dates back to the first half of the twentieth century, it was not until the 1970s that the study of monoid rings gained significant attention. A systematic treatment of ring-theoretical properties of monoid rings was initiated by R.~Gilmer and T.~Parker~\cite{rG74,GP74,GP75} in 1974. 
Since then monoid rings have received a substantial amount of consideration and have permeated through many fields under active research, including algebraic combinatorics~\cite{BCMP98}, discrete geometry~\cite{BG02}, and functional analysis~\cite{mA04}. During the last decades, monoid rings have also been studied from the point of view of factorization theory; see, for instance,~\cite{AJ15,AS97,hK01}. Gilmer in \cite{rG84} offers a comprehensive exposition on the advances of commutative semigroup ring theory until mid 1980s.

An integral domain is called atomic if every nonzero nonunit element it contains can be written as a product of irreducibles. Irreducible elements (sometimes called atoms) are the building blocks of atomicity and factorization theory. As a result, techniques to argue irreducibility are crucial in the development of factorization theory. Gauss's Lemma and Eisenstein's Criterion are two of the most elementary but effective tools to prove irreducibility in the context of polynomial rings. After reviewing some necessary terminology and background in Section~\ref{sec:background}, we dedicate Section~\ref{sec:Irreducibility Criteria} to extend Gauss's Lemma and Eisenstein's Criterion from the context of polynomial rings to that one of monoid rings.
\smallskip

An atomic monoid $M$ is called half-factorial provided that for all $x \in M$, any two factorizations of $x$ have the same number of irreducibles (counting repetitions). In addition, an integral domain is called half-factorial (or an HFD) if its multiplicative monoid is half-factorial. The concept of half-factoriality was first investigated by L.~Carlitz in the context of algebraic number fields; he proved that an algebraic number field is half-factorial if and only if its class group has size at most two~\cite{lC60}. Other-half-factoriality, on the other hand, is a dual version of half-factoriality, and it was introduced by J.~Coykendall and W.~Smith in~\cite{CS11}. 
\smallskip

Additive monoids of rationals have a wild atomic structure~\cite{fG19,fG17} and a complex arithmetic of factorizations~\cite{fG18a,GO19}. The monoid rings they determine have been explored in~\cite{ACHZ07}. In addition, examples of such monoid rings have also appeared in the past literature, including \cite[Section~1]{aG74}, \cite[Example~2.1]{AAZ90}, and more recently, \cite[Section~5]{CG19}. In the second part of this paper, which is Section~\ref{sec:factorization in monoid algebras}, we study half-factoriality and other-half-factoriality in the context of additive monoids of rationals and the monoid algebras they induce. We also determine which of these monoid algebras are Dedekind domains, Euclidean domains, PIDs, UFDs, and HFDs.
\bigskip

\section{Notation and Background}
\label{sec:background}

\subsection{General Notation} Throughout this paper, we let $\nn_0$ denote the set of all nonnegative integers, and we set $\nn := \nn_0 \setminus \{0\}$. If $a,b \in \zz$ and $a \le b$, then we let $\ldb a,b \rdb$ denote the interval of integers from $a$ to $b$, i.e.,
\[
	\ldb a,b \rdb := \{j \in \zz \mid a \le j \le b\}.
\]
For a subset $X$ of $\rr$, we set $X^\bullet := X \setminus \{0\}$. In addition, if $r \in \rr$, we define
\[
	X_{> r} := \{x \in X \mid x > r\} \quad \text{and} \quad X_{\ge r} := \{x \in X \mid x \ge r\}.
\]
If $q \in \qq_{> 0}$, then we denote the unique $m,n \in \nn$ such that $q = m/n$ and $\gcd(m,n)=1$ by $\mathsf{n}(q)$ and $\mathsf{d}(q)$, respectively.
\medskip

\subsection{Monoids} Within the scope of our exposition, a \emph{monoid} is defined to be a commutative and cancellative semigroup with an identity element. In addition, monoids here are written multiplicatively unless we specify otherwise. Let $M$ be a monoid. We let $U(M)$ denote the set of units (i.e., invertible elements) of $M$. When $U(M)$ consists of only the identity element, $M$ is said to be \emph{reduced}. On the other hand, $M$ is called \emph{torsion-free} if for all $x,y \in M$ and $n \in \nn$, the equality $x^n = y^n$ implies $x = y$. For $S \subseteq M$, we let $\langle S \rangle$ denote the submonoid of $M$ generated by $S$. Further basic definitions and concepts on commutative cancellative monoids can be found in~\cite[Chapter~2]{pG01}.

If $y,z \in M$, then $y$ \emph{divides} $z$ \emph{in} $M$ provided that there exists $x \in M$ such that $z = xy$; in this case we write $y \mid_M z$. Also, the elements $y$ and $z$ are called \emph{associates} if $y \mid_M z$ and $z \mid_M y$; in this case we write $y \simeq z$. An element $p \in M \setminus U(M)$ is said to be \emph{prime} when for all $x,y \in M$ with $p \mid_M xy$, either $p \mid_M x$ or $p \mid_M y$. If every element in $M \setminus U(M)$ can be written as a product of primes, then $M$ is called \emph{factorial}. In a factorial monoid every nonunit element can be uniquely written as a product of primes (up to permutation and associates). In addition, an element $a \in M \setminus U(M)$ is called an \emph{atom} if for any $x,y \in M$ such that $a = xy$ either $x \in U(M)$ or $y \in U(M)$. The set of all atoms of $M$ is denoted by $\mathcal{A}(M)$, and $M$ is said to be \emph{atomic} if every nonunit element of $M$ is a product of atoms. Since every prime element is clearly an atom, every factorial monoid is atomic.
\medskip

\subsection{Factorizations} Let $M$ be a monoid, and take $x \in M \setminus U(M)$. Suppose that for $m \in \nn$ and $a_1, \dots, a_m \in \mathcal{A}(M)$,
\begin{equation} \label{eq:factorization definition}
	x = a_1 \cdots a_m.
\end{equation}
Then the right-hand side of~(\ref{eq:factorization definition}) (treated as a formal product of atoms) is called a \emph{factorization} of $x$, and $m$ is called the \emph{length} of such a factorization. Two factorizations $a_1 \cdots a_m$ and $b_1 \cdots b_n$ of $x$ are considered to be equal provided that $m=n$ and that there exists a permutation $\sigma \in S_m$ such that $b_i \simeq a_{\sigma(i)}$ for every $i \in \ldb 1,m \rdb$. The set of all factorizations of $x$ is denoted by $\mathsf{Z}_M(x)$ or, simply, by $\mathsf{Z}(x)$. We then set
\[
	\mathsf{Z}(M) := \bigcup_{x \in M \setminus U(M)} \mathsf{Z}(x).
\]
For $z \in \mathsf{Z}(x)$, we let $|z|$ denote the length of $z$.
\medskip

\subsection{Monoid Rings} For an integral domain $R$, we let $R^\times$ denote the group of units of~$R$. We say that $R$ is \emph{atomic} if every nonzero nonunit element of $R$ can be written as a product of irreducibles (which are also called atoms).

Let $M$ be a reduced torsion-free monoid that is additively written. For an integral domain $R$, consider the set $R[X;M]$ comprising all maps $f \colon M \to R$ satisfying that
\[
	\{s \in M \mid f(s) \neq 0 \}
\]
is finite. We shall conveniently represent an element $f \in R[X;M]$ by 
\[
	f = \sum_{s \in M} f(s)X^s = \sum_{i=1}^n f(s_i)X^{s_i},
\]
where $s_1, \dots, s_n$ are those elements $s \in M$ satisfying that $f(s) \neq 0$. Addition and multiplication in $R[X;M]$ are defined as for polynomials, and we call the elements of $R[X;M]$ \emph{polynomial expressions}. Under these operations, $R[X;M]$ is a commutative ring, which is called the \emph{monoid ring of} $M$ \emph{over} $R$ or, simply, a \emph{monoid ring}. Following Gilmer~\cite{rG84}, we will write $R[M]$ instead of $R[X;M]$. Since $R$ is an integral domain, $R[M]$ is an integral domain~\cite[Theorem~8.1]{rG84} with set of units $R^\times$~\cite[Corollary~4.2]{GP74}. If~$F$ is a field, then we say that $F[M]$ is a \emph{monoid algebra}. Now suppose that the monoid $M$ is totally ordered. For $k \in \nn$, we say that
\[
	f = \alpha_1X^{q_1} + \dots + \alpha_k X^{q_k} \in R[M] \setminus \{0\}
\]
is written in \emph{canonical form} when the coefficient $\alpha_i$ is nonzero for every $i \in \ldb 1,k \rdb$ and $q_1 > \dots > q_k$. Observe that there is only one way to write $f$ in canonical form. We call $\deg(f) := q_1$ the \emph{degree} of $f$. In addition, $\alpha_1$ is called the \emph{leading coefficient} of $f$, and $\alpha_k$ is called the \emph{constant coefficient} of~$f$ provided that $q_k = 0$. As it is customary for polynomials, $f$ is called a \emph{monomial} when $k = 1$.

Suppose that $\psi \colon M \to M'$ is a monoid homomorphism, where $M$ and $M'$ are reduced torsion-free monoids. Also, let $\psi^* \colon R[M] \to R[M']$ be the ring homomorphism determined by the assignment $X^s \mapsto X^{\psi(s)}$. It follows from \cite[Theorem~7.2(2)]{rG84} that if $\psi$ is injective (resp., surjective), then $\psi^*$ is injective (resp., surjective). Let us recall the following easy observation.

\begin{remark} \label{rem:isomophism of monoid algebras}
	If $R$ is an integral domain and the monoids $M$ and $M'$ are isomorphic, then the monoid rings $R[M]$ and $R[M']$ are also isomorphic.
\end{remark}
\bigskip

\section{Irreducibility Criteria for Monoid Rings}
\label{sec:Irreducibility Criteria}

\subsection{Extended Gauss's Lemma} Our primary goal in this section is to offer extended versions of Gauss's Lemma and Eisenstein's Criterion for monoid rings.

Let $R$ be an integral domain and take $r_1, \dots, r_n \in R \setminus \{0\}$ for some $n \in \nn$. An element $r \in R$ is called a \emph{greatest common divisor} of $r_1, \dots, r_n$ if $r$ divides $r_i$ in $R$ for every $i \in \ldb 1,n \rdb$ and $r$ is divisible by each common divisor of $r_1, \dots, r_n$. Any two greatest common divisors of $r_1, \dots, r_n$ are associates in $R$. We let $\text{GCD}(r_1, \dots, r_n)$ denote the set of all greatest common divisors of $r_1, \dots, r_n$.

\begin{definition}
	An integral domain $R$ is called a \emph{GCD-domain} if any finite subset of $R \setminus \{0\}$ has a greatest common divisor in $R$.
\end{definition}

Let $M$ be a reduced torsion-free monoid, and let $R$ be an integral domain. Suppose that for the polynomial expression
\[
	f = \alpha_1 X^{q_1} + \dots + \alpha_k X^{q_k} \in R[M] \setminus \{0\}
\]
the exponents $q_1, \dots, q_k$ are pairwise distinct. Then GCD$(\alpha_1, \dots, \alpha_k)$ is called the \emph{content} of $f$ and is denoted by $\cont(f)$. If $\cont(f) = R^\times$, then $f$ is called \emph{primitive}. Notice that if $R$ is not a GCD-domain, then $\mathsf{c}(f)$ may be the empty set. It is clear that $\mathsf{c}(rf) = r \mathsf{c}(f)$ for all $r \in R \setminus \{0\}$ and $f \in R[M] \setminus \{0\}$. As for the case of polynomials, the following lemma holds.

\begin{lemma} \label{lem:content identity} 
	Let $M$ be a reduced torsion-free monoid, and let $R$ be a GCD-domain. If $f$ and $g$ are elements of $R[M] \setminus \{0\}$, then $\mathsf{c}(fg) = \mathsf{c}(f) \mathsf{c}(g)$.
\end{lemma}

\begin{proof}
	Since $R$ is a GCD-domain, there exist primitive polynomial expressions~$f_1$ and~$g_1$ in $R[M]$ such that $f = \mathsf{c}(f) f_1$ and $g = \mathsf{c}(g) g_1$. Because $M$ is a torsion-free monoid, it follows from~\cite[Proposition~4.6]{GP74} that the element $f_1 g_1$ is primitive in $R[M]$. Therefore $\mathsf{c}(f_1 g_1) = R^\times$. As a consequence, we find that
	\[
		\mathsf{c}(fg) = \mathsf{c} \big( \mathsf{c}(f) f_1 \mathsf{c}(g) g_1 \big) = \mathsf{c}(f) \mathsf{c}(g) \mathsf{c}(f_1 g_1) = \mathsf{c}(f) \mathsf{c}(g),
	\]
	as desired.
\end{proof}

Let $F$ denote the field of fractions of a GCD-domain $R$. Gauss's Lemma states that a non-constant polynomial $f$ with coefficients in $R$ is irreducible in $R[X]$ if and only if it is irreducible in $F[X]$ and primitive in $R[X]$. Now we extend Gauss's Lemma to the context of monoid rings.

\begin{theorem}[Extended Gauss's Lemma]
	Let $M$ be a reduced torsion-free monoid, and let $R$ be a GCD-domain with field of fractions $F$. Then an element $f \in R[M] \setminus R$ is irreducible in $R[M]$ if and only if $f$ is irreducible in $F[M]$ and primitive in $R[M]$.
\end{theorem}

\begin{proof}
	For the direct implication, suppose that $f$ is irreducible in $R[M]$. If $r \in \mathsf{c}(f)$, then there exists $g \in R[M] \setminus R$ such that $f = rg$. Because $R[M]^\times \subset R$, the element $g$ is not a unit of $R[M]$. As $f$ is irreducible in $R[M]$, one finds that $r \in R[M]^\times = R^\times$. So $\mathsf{c}(f) = R^\times$, which implies that $f$ is primitive in $R[M]$. To argue that $f$ is irreducible in $F[M]$, take $g_1, g_2 \in F[M]$ such that $f = g_1 g_2$. Since $R$ is a GCD-domain, there exist nonzero elements $a_1, a_2, b_1, b_2 \in R$ such that both
	\[
		h_1 := \frac{a_1}{b_1} g_1 \quad \text{and} \quad h_2 := \frac{a_2}{b_2} g_2
	\]
	are primitive elements of $R[M]$. Clearly, $a_1 a_2 f = b_1 b_2 h_1 h_2$. This, along with Lemma~\ref{lem:content identity}, implies that
	\[
		a_1 a_2 R^\times = a_1 a_2 \mathsf{c}(f) = \mathsf{c}(a_1 a_2 f) = \mathsf{c}(b_1 b_2 h_1 h_2) = b_1 b_2 \mathsf{c}(h_1) \mathsf{c}(h_2) = b_1 b_2 R^\times.
	\]
	Then $\frac{a_1 a_2}{b_1 b_2} \in R^\times$ and, as a consequence, $\frac{a_1 a_2}{b_1 b_2} f = h_1 h_2$ is irreducible in $R[M]$. Thus, either $h_1 \in R[M]^\times = R^\times$ or $h_2 \in R[M]^\times = R^\times$. This, in turn, implies that either~$g_1$ or~$g_2$ belongs to $F^\times = F[M]^\times$. Hence $f$ is irreducible in $F[M]$.
	
	Tor argue the reverse implication, suppose that $f$ is irreducible in $F[M]$ and primitive in $R[M]$. Then take elements $g_1$ and $g_2 \in R[M]$ such that $f = g_1 g_2$. Since $f$ is irreducible in $F[M]$, either $g_1 \in F[M]^\times = F^\times$ or $g_2 \in F[M]^\times = F^\times$. This, along with the fact that $R[M] \cap F^\times = R \setminus \{0\}$, implies that either $g_1 \in \mathsf{c}(f)$ or $g_2 \in \mathsf{c}(f)$. As $\mathsf{c}(f) = R^\times = R[M]^\times$, either $g_1$ or $g_2$ belongs to $R[M]^\times$. As a result, $f$ is irreducible in $R[M]$, which concludes the proof.
\end{proof}
\medskip

\subsection{Extended Eisenstein's Criterion} It is hardly debatable that Eisenstein's Criterion is one of the most popular and useful criteria to argue the irreducibility of certain polynomials. Now we proceed to offer an extended version of Eisenstein's Criterion for monoid rings.

\begin{prop}[Extended Eisenstein's Criterion] \label{prop:Eisenstein Criterion for PA}
	Let $M$ be a reduced totally-ordered torsion-free monoid, and let $R$ be an integral domain. Suppose that the element
	\[
		f = \alpha_n X^{q_n} + \dots + \alpha_1 X^{q_1} + \alpha_0 \in R[M] \setminus \{0\},
	\]
	written in canonical form, is primitive. If there exists a prime ideal $P$ of $R$ satisfying the conditions
	\begin{enumerate}
		\item $\alpha_n \notin P$,
		\vspace{3pt}
		\item $\alpha_j \in P$ for every $j \in \ldb 0,n-1 \rdb$, and
		\vspace{3pt}
		\item $\alpha_0 \notin P^2$,
	\end{enumerate}
	then $f$ is irreducible in $R[M]$.
\end{prop}

\begin{proof}
	We let $\bar{R}$ denote the quotient $R/P$ and, for any $h \in R[M]$, we let~ $\bar{h}$ denote the image of $h$ under the natural surjection $R[M] \to \bar{R}[M]$, i.e., $\bar{h}$ is the result of reducing the coefficients of $h$ modulo $P$. To argue that $f$ is irreducible suppose, by way of contradiction, that $f = g_1 g_2$ for some nonzero nonunit elements $g_1$ and $g_2$ of $R[M]$. As~$f$ is primitive, $g_1 \notin R$ and $g_2 \notin R$. By the condition~(2) in the statement, one obtains that $\bar{g}_1 \bar{g}_2 = \bar{f} = \bar{\alpha}_n X^{q_n}$. Thus, both $\bar{g}_1$ and $\bar{g}_2$ are monomials. This, along with the fact that none of the leading coefficients of $g_1$ and $g_2$ are in $P$ (because $\alpha_n \notin P$), implies that the constant coefficients of both $g_1$ and $g_2$ are in $P$. As a result, the constant coefficient $\alpha_0$ of $f$ must belong to $P^2$, which is a contradiction.
\end{proof}

\begin{cor} \label{cor:irreducible polynomials of any degree}
	Let $M$ be a reduced totally-ordered torsion-free monoid, and let $R$ be an integral domain containing a prime element. Then for each $q \in M^\bullet$, there exists an irreducible polynomial expression in $R[M]$ of degree $q$.
\end{cor}

\begin{proof}
	Let $p$ be a prime element of $R$. It suffices to verify that, for any $q \in M^\bullet$, the element $f := X^q + p \in R[M]$ is irreducible. Indeed, this is an immediate consequence of Proposition~\ref{prop:Eisenstein Criterion for PA} once we take $P := (p)$. 
\end{proof}

In Corollary~\ref{cor:irreducible polynomials of any degree}, the integral domain $R$ is required to contain a prime element. This condition is not superfluous, as the next example illustrates.

\begin{example}
	For a prime number $p$, consider the monoid algebra $\ff_p[M]$, where $M$ is the submonoid $\langle 1/p^n \mid n \in \nn \rangle$ of $(\qq_{\ge 0},+)$ and $\ff_p$ is a finite field of characteristic $p$. It is clear that $M$ is a reduced totally-ordered torsion-free monoid. Now let
	\[
		f := \alpha_1 X^{q_1} + \dots + \alpha_n X^{q_n}
	\]
	be an element of $\ff_p[M] \setminus \ff_p$ written in canonical form. As $\ff_p$ is a perfect field of characteristic $p$, the Frobenius homomorphism $x \mapsto x^p$ is surjective and, therefore, for each $i \in \ldb 1,n \rdb$ there exists $\beta_i \in \ff_p$ with $\alpha_i = \beta_i^p$. On the other hand, it is clear that $q_i/p \in M$ for every $i \in \ldb 1,n \rdb$. As
	\[
		f = \alpha_1 X^{q_1} + \dots + \alpha_n X^{q_n} = \big( \beta_1 X^{q_1/p} + \dots + \beta_n X^{q_n/p} \big)^p,
	\]
	the polynomial expression $f$ is not irreducible in $\ff_p[M]$. Hence the monoid algebra $\ff_p[M]$ does not contain irreducible elements. Clearly, the field $\ff_p$ is an integral domain containing no prime elements.
\end{example}
\medskip

\section{Factorizations in Monoid Algebras}
\label{sec:factorization in monoid algebras}

A \emph{numerical semigroup} is a submonoid $N$ of $(\nn_0,+)$ whose complement is finite, i.e., $|\nn_0 \setminus N| < \infty$. Numerical semigroups are finitely generated and, therefore, atomic. However, the only factorial numerical semigroup is $(\nn_0,+)$. For an introduction to numerical semigroups, see \cite{GR09}, and for some of their many applications, see~\cite{AG16}. A \emph{Puiseux monoid}, on the other hand, is an additive submonoid of $(\qq_{\ge 0},+)$. Albeit Puiseux monoids are natural generalizations of numerical semigroups, the former are not necessarily finitely generated or atomic; for example, consider $\langle 1/2^n \mid n \in \nn \rangle$. The factorization structure of Puiseux monoids have been compared with that of other well-studied atomic monoids in~\cite{fG18} and, more recently, in~\cite{CGG19}. In this section, we determine the Puiseux monoids whose monoid algebras are Dedekind domains, Euclidean domains, PIDs, UFDs, or HFDs.

\begin{definition}
	An atomic monoid $M$ is \emph{half-factorial} (or an \emph{HF-monoid}) if for all $x \in M \setminus U(M)$ and $z, z' \in \mathsf{Z}(x)$ the equality $|z| = |z'|$ holds. An integral domain is \emph{half-factorial} (or an \emph{HFD}) if its multiplicative monoid is an HF-monoid.
\end{definition}

Clearly, half-factoriality is a relaxed version of being a factorial monoid or a UFD. Although the concept of half-factoriality was first considered by Carlitz in his study of algebraic number fields~\cite{lC60}, it was A.~Zaks who coined the term ``half-factorial domain"~\cite{aZ76}.

\begin{definition}
	An atomic monoid $M$ is \emph{other-half-factorial} (or an \emph{OHF-monoid}) if for all $x \in M \setminus U(M)$ and $z, z' \in \mathsf{Z}(x)$ the equality $|z| = |z'|$ implies that $z = z'$.
\end{definition}

Observe that other-half-factoriality is somehow a dual version of half-factoriality. Although an integral domain is a UFD if and only if its multiplicative monoid is an OHF-monoid~\cite[Corollary~2.11]{CS11}, OHF-monoids are not always factorial or half-factorial, even in the class of Puiseux monoids.
\medskip

\begin{prop} \label{prop:HF and OHF PM characterization}
	For a nontrivial atomic Puiseux monoid $M$, the following conditions hold.
	\begin{enumerate}
		\item $M$ is an HF-monoid if and only if $M$ is factorial.
		\vspace{3pt}
		\item $M$ is an OHF-monoid if and only if $|\mathcal{A}(M)| \le 2$.
	\end{enumerate}
\end{prop}

\begin{proof}
	For the direct implication of~(1), suppose that $M$ is an HF-monoid. Since $M$ is an atomic nontrivial Puiseux monoid, $\mathcal{A}(M)$ is not empty. Let $a_1$ and $a_2$ be two atoms of $M$. Then $z_1 := \mathsf{n}(a_2) \mathsf{d}(a_1) a_1$ and $z_2 :=  \mathsf{n}(a_1) \mathsf{d}(a_2) a_2$ are two factorizations of the element $\mathsf{n}(a_1) \mathsf{n}(a_2) \in M$. Because $M$ is an HF-monoid, $|z_1| = |z_2|$ and so
	\[
		\mathsf{n}(a_2) \mathsf{d}(a_1) = \mathsf{n}(a_1)  \mathsf{d}(a_2).
	\]
	Therefore $a_1 = a_2$, and then $M$ contains only one atom. Hence $M \cong (\nn_0,+)$ and, as a result, $M$ is factorial. The reverse implication of~(1) is trivial.
	
	To prove the direct implication of~(2), assume that $M$ is an OHF-monoid. If $M$ is factorial, then $M \cong (\nn_0,+)$, and we are done. Then suppose that $M$ is not factorial. In this case, $|\mathcal{A}(M)| \ge 2$. Assume, by way of contradiction, that $|\mathcal{A}(M)| \ge 3$. Take $a_1, a_2, a_3 \in \mathcal{A}(M)$ satisfying that $a_1 < a_2 < a_3$. Let $d = \mathsf{d}(a_1) \mathsf{d}(a_2) \mathsf{d}(a_3)$, and set $a'_i = d a_i$ for each $i \in \ldb 1,3 \rdb$. Since $a'_1, a'_2$, and $a'_3$ are integers satisfying that $a'_1 < a'_2 < a'_3$, there exist $m,n \in \nn$ such that
	\begin{equation} \label{eq:OHF}
		m(a'_2 - a'_1) = n(a'_3 - a'_2).
	\end{equation}
	Clearly, $z_1 := ma_1 + na_3$ and $z_2 := (m+n)a_2$ are two distinct factorizations in $\mathsf{Z}(M)$ satisfying that $|z_1| = m+n = |z_2|$. In addition, after dividing both sides of the equality~(\ref{eq:OHF}) by $d$, one obtains that
	\[
		ma_1 + na_3 = (m+n)a_2,
	\]
	which means that $z_1$ and $z_2$ are factorizations of the same element. However, this contradicts that $M$ is an OHF-monoid. Hence $|\mathcal{A}(M)| \le 2$, as desired. For the reverse implication of~(2), suppose that $|\mathcal{A}(M)| \le 2$. By~\cite[Proposition~3.2]{fG17}, $M$ is isomorphic to a numerical semigroup $N$. As $N$ is generated by at most two elements, either $N = (\nn_0,+)$ or $N = \langle a, b \rangle$ for $a, b \in \nn_{\ge 2}$ with $\gcd(a,b) = 1$. If $N = (\nn_0,+)$, then~$N$ is factorial and, in particular, an OHF-monoid. On the other hand, if $N = \langle a, b \rangle$, then it is an OHF-monoid by~\cite[Example~2.13]{CS11}.
\end{proof}
\medskip

In~\cite[Theorem~8.4]{GP74} Gilmer and Parker characterize the monoid algebras that are Dedekind domains, Euclidean domains, or PIDs. We conclude this section extending such a characterization in the case where the exponent monoids are Puiseux monoids.

\begin{theorem} 
	For a nontrivial Puiseux monoid $M$ and a field $F$, the following conditions are equivalent:
	\begin{enumerate}
		\item $F[M]$ is a Euclidean domain;
		\vspace{3pt}
		\item $F[M]$ is a PID;
		\vspace{3pt}
		\item $F[M]$ is a UFD;
		\vspace{3pt}
		\item $F[M]$ is an HFD;
		\vspace{3pt}
		\item $M \cong (\nn_0,+)$;
		\vspace{3pt}
		\item $F[M]$ is a Dedekind domain.
	\end{enumerate}
\end{theorem}

\begin{proof}
	It is well known that every Euclidean domain is a PID, and every PID is a UFD. Therefore condition~(1) implies condition~(2), and condition~(2) implies condition~(3). In addition, it is clear that every UFD is an HFD, and so condition~(3) implies condition~(4). As Puiseux monoids are torsion-free,~\cite[Proposition~1.4]{hK01} ensures that $M$ is an HF-monoid when $F[M]$ is an HFD. This, along with Proposition~\ref{prop:HF and OHF PM characterization}(1), guarantees that $M \cong (\nn_0,+)$ provided that $F[M]$ is an HFD. Thus, condition~(4) implies condition~(5). Now notice that if condition~(5) holds, then $F[M] \cong F[\nn_0] = F[X]$ (by Remark~\ref{rem:isomophism of monoid algebras}) is a Euclidean domain, which is condition~(1). Then we have argued that the first five conditions are equivalent.
	
	To include~(6) in the set of already-established equivalent conditions, observe that condition~(2) implies condition~(6) because every PID is a Dedekind domain. On the other hand, suppose that the monoid algebra $F[M]$ is a Dedekind domain. Then the fact that $M$ is torsion-free, along with~\cite[Theorem~8.4]{GP74}, implies that $M \cong (\nn_0,+)$. Hence condition~(6) implies condition~(5), which completes the proof.
\end{proof}
\medskip

\section*{Acknowledgments}

\noindent While working on this paper, the author was supported by the NSF-AGEP Fellowship and the UC Dissertation Year Fellowship. The author would like to thank an anonymous referee, whose suggestions help to simplify and improve the initially-submitted version of this paper.

\end{document}